\documentclass[12pt]{amsart}
\usepackage[margin=1.0in]{geometry} 
\usepackage[T1]{fontenc}
\usepackage{lmodern}
\usepackage{amsmath,amssymb,amsthm}
\usepackage{microtype}
\usepackage{xcolor}
\usepackage[
unicode=true,
pdfusetitle,
bookmarksnumbered=true,
bookmarksopen=true,
breaklinks=true,
colorlinks=true,
linkcolor=blue,
citecolor=blue,
urlcolor=blue
]{hyperref}
\usepackage{bookmark}

\makeatletter
\@namedef{subjclassname@2020}{%
 \textup{2020} Mathematics Subject Classification}
\makeatother

\numberwithin{equation}{section}

\theoremstyle{plain}
\newtheorem{theorem}{Theorem}[section]
\newtheorem{lemma}[theorem]{Lemma}
\newtheorem{proposition}[theorem]{Proposition}
\newtheorem{corollary}[theorem]{Corollary}
\newtheorem{conjecture}[theorem]{Conjecture}

\theoremstyle{remark}
\newtheorem{remark}[theorem]{Remark}

\newcommand{\F}{\mathbb F}
\newcommand{\ord}{\operatorname{ord}}
\title[Sumsets and progressions in multiplicative subgroups]
{Sumsets and generalized arithmetic progressions in multiplicative subgroups}
\author{Albert Cochrane}
\address{Department of Mathematics\\
         Kansas State University\\
         Manhattan, KS, USA}
\email{albertc@ksu.edu}
\date{}

\begin{document}
\begin{abstract}
Let $q=p^f$, and let $A\leq\F_q^\times$ be a multiplicative subgroup with $\F_p(A)=\F_q$.  We prove that a proper subgroup $A$ is a generalized arithmetic progression (GAP) if and only if $|A| \in \{1, 2, 4\}$, and we determine when the full group $\F_q^\times$ is a GAP. For certain families of subgroups, we obtain the stronger conclusion that $A$ is additively irreducible. In particular, if $|A|>4$ and $p^e\equiv-1\pmod{|A|}$ for some $e\ge1$, then $A$ admits no nontrivial sumset decomposition. We also prove that every $c \neq 0$ has fewer than $|A|/2$ representations as a sum (or difference) of two elements of $A$ whenever $[\F_q^\times:A] \ge3$ and $|A| \ge 5$, which may be of independent interest.
\end{abstract}

\maketitle

\section{Introduction}
\label{sec:intro}

A recurring theme in additive combinatorics is that strong additive and
multiplicative structure should rarely coexist.  Multiplicative subgroups of
finite fields provide a natural setting in which to study this principle.
Let $q=p^f$, where $p$ is prime, let $k\mid q-1$, and write
\[
       A_k=\{x^k:x\in\F_q^\times\}\leq\F_q^\times,
       \qquad |A_k|=t=\frac{q-1}{k}.
\]
We regard $A_k$ as a subset of the additive group of $\F_q$ and ask when it
can be written as a nontrivial sumset or as a generalized arithmetic
progression.

For finite nonempty sets $S$ and $T$, we write the sumset $S+T=\{s+t:s\in S,\ t\in T\}$. An additive decomposition $A=S+T$ is \emph{nontrivial} if $|S|,|T|\geq2$; in such a
decomposition, $S$ and $T$ are additive summands of $A$.  We call $A$
\emph{additively irreducible} if it admits no nontrivial additive
decomposition.

For $r\geq1$, a rank-$r$ \emph{generalized arithmetic progression} (GAP) in
an abelian group is a set
\[
       P=a+[0,L_1-1]d_1+\cdots+[0,L_r-1]d_r,
\]
where each $L_i\geq2$ and $d_i\ne0$.  It is \emph{proper} when the parameter map
from the defining box is injective, equivalently when
$|P|=L_1\cdots L_r$.  We also regard a singleton as a rank-zero GAP. 

\subsection{Background and relation to prior work}
\label{subsec:priorwork}

The rank-one case of the GAP problem is classical.  Chowla, Mann, and
Straus \cite{cms} classified the multiplicative subgroups of prime fields
that are arithmetic progressions.  The classification for GAPs over prime
fields was proved in \cite{cochrane2026}, where it follows from directness
results for additive decompositions; related doubling questions were
considered in \cite{ccs}.

Additive decompositions of multiplicative subgroups have been studied
extensively.  S\'ark\"ozy conjectured that the quadratic residues modulo
sufficiently large primes are additively irreducible \cite{sa}. Related decomposition problems for
multiplicative subgroups, quadratic residues, and primitive roots have since been
considered in prime and general finite fields; see, among others,
\cite{bks,chenxi,chenyan,dart1,shk1,shkredov,wu}. Kalmynin
resolved this conjecture \cite[Theorem~3]{kalmynin} and proved
\cite[Theorem~2]{kalmynin} that, for every proper subgroup
$A_k<\F_p^\times$, a nontrivial decomposition $A_k=S+T$
must satisfy $|S|=|T|=\sqrt{|A_k|}$.

Building on the polynomial method introduced by Hanson and Petridis \cite{hansonpetridis} and Kalmynin's subsequent work, Rudnev and Tyrrell \cite[Theorem 1.1]{ru} recently completed the prime field classification of reducible subgroups: a proper subgroup $A_k < \F_p^\times$ admits a nontrivial decomposition if and only if $|A_k|= 4$.
In particular, Kalmynin's theorem gives the prime-field case of
Theorem~\ref{thm:no-two-point} below, and the result of Rudnev and Tyrrell gives the prime-field case of Theorem~\ref{thm:proper-classification}.

Over general finite fields, however, the corresponding fixed-index irreducibility problem remains open and is commonly referred to as the generalized S\'ark\"ozy conjecture; see, for example,
\cite{kimyipyoo,yip,yip2}.

\begin{conjecture}
\label{conj:sarkozy}
Let $k\geq2$ be fixed.  Then, for all sufficiently large prime powers
$q\equiv1\pmod k$, the subgroup $A_k<\F_q^\times$ admits no nontrivial
additive decomposition
\[
       A_k=S+T,
       \qquad |S|,|T|\geq2.
\]
\end{conjecture}

Shparlinski initiated the study of additive decompositions of multiplicative subgroups over general finite fields. In particular, \cite[Theorem~7]{shparlinski} shows that for every $\varepsilon>0$, if $|A_k|\geq q^{3/4+\varepsilon}$, then any nontrivial decomposition $A_k=S+T$ satisfies $\max\{|S|,|T|\}\ll_{\varepsilon}q^{1/2}$. Yip obtained several sufficient conditions excluding representations of the
form $A_k=B+B$ over general finite fields \cite[Corollary~1.4]{yip}, as well
as corresponding results for restricted sumsets
\cite[Theorems~1.1 and 1.3]{yip2}.  A multiplicative analogue involving
shifted multiplicative subgroups is considered in \cite{kimyipyoo}. 

These additive questions are unchanged if the ambient field is replaced by
the field generated by the subgroup.  Accordingly, throughout the paper we may
work in the minimal ambient field,
\[
       \F_p(A_k)=\F_q,
       \qquad k=[\F_q^\times:A_k].
\]
By Lemma~\ref{lem:minfield}, this convention changes neither additive
decompositions nor GAP representations.  Thus
$f=\ord_t(p)$ when $t>1$, while for $t=1$ we take $q=p$. 

\subsection{Main results}
\label{subsec:main-results}

Our first theorem classifies the proper multiplicative subgroups that are
GAPs.

\begin{theorem}
\label{thm:proper-classification}
Let $q=p^f$ and $A_k<\F_q^\times$ satisfy $\F_p(A_k)=\F_q$.  Then $A_k$ is a GAP if and
only if
\[
       |A_k|\in\{1,2,4\}.
\]
\end{theorem}

For $f=1$ and positive rank, this is the result of
\cite{cochrane2026}; the $f\geq 2$ case is the interest here.  The
positive examples, with $i^2=-1$ in the order-four case, are
\[
       \{1\},\qquad
       \{\pm1\}=-1+\{0,2\},\qquad
       \{\pm1,\pm i\}=-1+\{0,1-i\}+\{0,1+i\}.
\]
The restriction to proper subgroups is essential, as full multiplicative
groups behave differently.

\begin{theorem}
\label{thm:full-groups}
Let $q=p^f$.  If $p$ is odd, then $\F_q^\times$ is a GAP.  For $f=1$ it is
an arithmetic progression.  For $f>1$ it is not a proper GAP of any rank,
but it has an improper GAP representation of rank at most $2f$.  If $p=2$,
then $\F_q^\times$ is a GAP if and only if $q=2$.
\end{theorem}

The key local obstruction behind Theorem~\ref{thm:proper-classification} is
the following restriction on additive summands.

\begin{theorem}
\label{thm:no-two-point}
Let $A_k<\F_q^\times$ satisfy $\F_p(A_k)=\F_q$ and have order $t\geq5$.
If
\[
       A_k=S+T,
       \qquad |S|,|T|\geq2,
\]
then
\[
       3\leq |S|,|T|\leq \frac t2,
\]
and for $k\geq 3$ the upper bound is strict. 
\end{theorem}

The lower bound says precisely that $A_k$ has no two-point additive summand. Since every GAP with at least two elements has a two-point additive summand
(Lemma~\ref{lem:gap-two-point}), this shows that no proper subgroup of order at least $5$ is a GAP.  For subgroups satisfying certain congruence conditions, we can rule out all nontrivial
additive decompositions. This statement is insensitive to the ambient field, so the minimal-field convention is not needed.

\begin{theorem}
\label{thm:frobenius-weight-main}
Let $A_k\leq\F_q^\times$ have order $t>1$.  Suppose that
$p^e\equiv r\pmod t$ for an integer $r$, where $e\geq1$, $r\ne1$, and
$p\nmid r$.  Put
\[
M(r)=
\begin{cases}
       r-1,&r\geq2,\\
       2|r|,&r\leq-1.
\end{cases}
\]
Then
\[
       |(A_k-a)\cap(A_k-b)|\leq M(r)
       \qquad(a,b\in\F_q,\ a\ne b).
\]
Every nontrivial decomposition $A_k=S+T$ therefore satisfies
$|S|,|T|\leq M(r)$.  Consequently, if $t>M(r)^2$, then $A_k$ is
additively irreducible.
\end{theorem}

The especially simple case $r=-1$ gives a useful concrete family of additively irreducible multiplicative subgroups.

\begin{corollary}
\label{cor:unit-circle}
Let $A_k\leq\F_q^\times$ have order $t>4$.  If
\[
       p^e\equiv-1\pmod t
\]
for some $e\geq1$, then
\[
       |(A_k-a)\cap(A_k-b)|\leq2
       \qquad(a,b\in\F_q,\ a\ne b).
\]
In particular, $A_k$ is additively irreducible.
\end{corollary}
This produces infinite families of irreducible subgroups in every characteristic.  Let $Q=p^e\geq4$
and let
\[
\mu_{Q+1}=\ker\!\left(
N_{\F_{Q^2}/\F_Q}:\F_{Q^2}^{\times}\to\F_Q^\times
\right)
\]
be the norm-one subgroup.
 Since $|\mu_{Q+1}|=Q+1$ and
$Q\equiv-1\pmod{Q+1}$, Corollary~\ref{cor:unit-circle} shows that
$\mu_{Q+1}$ is additively irreducible and that any two of its distinct
additive translates meet in at most two points.  Thus one obtains
additively irreducible subgroups with order roughly on the scale of $\sqrt{q}$.

To prove Theorem~\ref{thm:no-two-point}, we use quantitative bounds on sums and differences. For $A\subseteq\F_q$ and
$c,d\in\F_q$, we write the sum and difference representation counting functions as
\[
\begin{aligned}
       r_{A+A}(c)&=\#\{(x,y)\in A^2:x+y=c\},\\
       r_{A-A}(d)&=\#\{(x,y)\in A^2:x-y=d\}
       =|A\cap(A-d)|.
\end{aligned}
\]

\begin{theorem}
\label{thm:half-rep}
Let $A_k<\F_q^\times$ satisfy $\F_p(A_k)=\F_q$, with $k\geq3$ and
$t=|A_k|\geq5$.  Then
\[
       r_{A_k+A_k}(c)<\frac t2
       \qquad(c\in\F_q^\times).
\]
\end{theorem}

From this we also obtain the following strict difference
bound when $k\geq3$.

\begin{corollary}
\label{cor:diff-bound}
Let $A_k<\F_q^\times$ satisfy $\F_p(A_k)=\F_q$, with $k\geq3$ and
$t=|A_k|\geq5$.  Then, for every $d\ne0$,
\[
       r_{A_k-A_k}(d)<\frac t2.
\]
\end{corollary}

The paper is organized as follows.  Section~\ref{sec:prelim} contains preliminary lemmas on the minimal-field reduction and additive decompositions, as well as the proof of
Theorem~\ref{thm:frobenius-weight-main}. Section~\ref{sec:half} converts congruences modulo $t$ into bounds on $r_{A_k+A_k}$ via a short auxiliary-polynomial argument, and combines this with estimates of Garcia and Voloch on Fermat curves to prove Theorem~\ref{thm:half-rep}.
Section~\ref{sec:no-two-point} proves Corollary~\ref{cor:diff-bound} and treats the $k = 2$ case before proving Theorem~\ref{thm:no-two-point}. Section~\ref{sec:proper} proves Theorem~\ref{thm:proper-classification}, and Section~\ref{sec:full} treats the full group via a punctured vector space argument and proves Theorem~\ref{thm:full-groups}.

\section{Additive summands and translate intersections}
\label{sec:prelim}

We first record the elementary subfield reduction stated in the introduction.

\begin{lemma}
\label{lem:minfield}
Suppose $A\subseteq K\subseteq\F_q$, where $K$ is a subfield.
Every additive decomposition of $A$ may, after translating the summands, be
realized inside $K$.  Likewise, every GAP representation of $A$ has its
base point and all its directions in $K$.
\end{lemma}

\begin{proof}
Suppose $A=S_1+\cdots+S_\ell$.  Choose $s_i^0\in S_i$ for $1 \le i \le \ell$.  For $s,s'\in S_i$,
both $s+\sum_{j\ne i}s_j^0$ and $s'+\sum_{j\ne i}s_j^0$ lie in
$A\subseteq K$, so
\[
       s-s'\in K.
\]
Thus $S_i-s_i^0\subseteq K$.  Since $a_0=\sum_i s_i^0\in A\subseteq K$,
we may absorb $a_0$ into one translated summand and obtain a decomposition
inside $K$.

If
\[
       A=a+\sum_i[0,L_i-1]d_i\subseteq K,
\]
then $a\in A\subseteq K$, and $a+d_i\in A\subseteq K$ for each $i$, hence $d_i\in K$.
\end{proof}

\begin{lemma}
\label{lem:gap-two-point}
Every GAP with at least two elements has a two-point additive summand, whether
or not its parametrization is proper.
\end{lemma}

\begin{proof}
Choose one component progression and write
\[
       [0,L-1]d=\{0,d\}+[0,L-2]d,
       \qquad L\ge2.
\]
This gives a two-point additive summand of the GAP.
\end{proof}

The difference representation function has the useful intersection form $r_{A-A}(d)=|A\cap(A-d)|$, and if $A=-A$, the substitution $(x,y)\mapsto(x,-y)$ gives $r_{A-A}(d)=r_{A+A}(d)$. We use this
observation repeatedly in Section \ref{sec:no-two-point}.

We will also use the following elementary criterion for excluding summands of size $2$.

\begin{proposition}
\label{prop:edge-obstruction}
Let $A\subseteq\F_q$ be finite.  If
\[
       r_{A-A}(d)<\frac{|A|}{2}
       \qquad(d\ne0),
\]
then $A$ has no two-point additive summand.
\end{proposition}

\begin{proof}
Suppose
\[
       A=\{a,a+d\}+T,
       \qquad d\ne0.
\]
Setting $T'=T+a$ gives $A=\{0,d\}+T'$.  Since $T'\subseteq A$ and
$T'+d\subseteq A$, we have $T'\subseteq A\cap(A-d)$, while plainly
$|A|\le2|T'|$.  Hence
\[
       r_{A-A}(d)\ge|T'|\ge\frac{|A|}2,
\]
a contradiction.
\end{proof}

The same difference representation function also controls arbitrary two-summand
decompositions.  For distinct $a,b$ one has
\[
       |(A-a)\cap(A-b)|=r_{A-A}(a-b),
\]
which yields the following general obstruction.

\begin{lemma}
\label{lem:shifted-intersection-obstruction}
Let $A$ be a finite subset of an abelian group, and suppose that
\[
       |(A-a)\cap(A-b)|\le M
       \qquad(a\ne b).
\]
Every decomposition $A=S+T$ with $|S|,|T|\ge2$ satisfies
$|S|,|T|\le M$.  Consequently, if $|A|>M^2$, then no such decomposition
exists.
\end{lemma}

\begin{proof}
For distinct $a,b\in S$ one has
$T\subseteq(A-a)\cap(A-b)$, so $|T|\le M$, and interchanging the summands gives
$|S|\le M$.  The final assertion follows from
$|A|=|S+T|\le|S||T|\le M^2$.
\end{proof}

\begin{proof}[Proof of Theorem~\ref{thm:frobenius-weight-main}]
Put $P=p^e$.  Since $A_k$ is the group of $t$-th roots of unity in
$\F_q$ and $P\equiv r\pmod t$, the Frobenius automorphism $u\mapsto u^P$
restricts to the power map $u\mapsto u^r$ on $A_k$, where for $r<0$ we write
$u^r=(u^{-1})^{|r|}$.

We first prove the intersection bound.
Let $a\ne b$ and $x\in(A_k-a)\cap(A_k-b)$.  Since
$x+a,x+b\in A_k$ and the automorphism $u\mapsto u^P$ is additive,
\[
       (x+a)^r=(x+a)^P=x^P+a^P,
       \qquad
       (x+b)^r=(x+b)^P=x^P+b^P,
\]
whence
\begin{equation}
\label{eq:frobenius-two-shift}
       (x+a)^r-(x+b)^r=(a-b)^P.
\end{equation}
If $r\ge2$, the $x^r$-terms cancel, and
\eqref{eq:frobenius-two-shift} is a polynomial equation in $x$ of
degree $r-1$ with leading coefficient $r(a-b)\ne0$, as $p\nmid r$.
If $r\le-1$, multiply \eqref{eq:frobenius-two-shift} by
$(x+a)^{|r|}(x+b)^{|r|}$, which is nonzero; the
result is a polynomial equation in $x$ of degree $2|r|$ with leading
coefficient $(a-b)^P\ne0$.  Hence there are at most $M(r) = r-1$ solutions when
$r\ge2$, and at most $M(r) = 2|r|$ solutions when $r\le-1$.  Thus
\[
       |(A_k-a)\cap(A_k-b)|\le M(r)
       \qquad(a\ne b).
\]
It remains to deduce the decomposition statements. Lemma~\ref{lem:shifted-intersection-obstruction} gives
$|S|,|T|\le M(r)$ in every nontrivial decomposition $A_k=S+T$.  If
$t>M(r)^2$, the same lemma gives the additive irreducibility of $A_k$.
\end{proof}

\begin{proof}[Proof of Corollary~\ref{cor:unit-circle}]
Take $r=-1$ in Theorem~\ref{thm:frobenius-weight-main}.  Then $M(-1)=2$
and $t>4=M(-1)^2$.
\end{proof}

\begin{remark}
\label{rem:r-two}
If $p$ is odd and $p^e\equiv2\pmod t$ for some $e\ge1$, then taking $r=2$
makes \eqref{eq:frobenius-two-shift} linear in $x$: distinct translates of
$A_k$ meet in at most one point, and
Lemma~\ref{lem:shifted-intersection-obstruction} with $M=1$ excludes
$A_k=S+T$ with $|S|,|T|\ge2$ for every $t>1$.
\end{remark}

The next section further develops the interaction between Frobenius automorphisms and
the sum representation function $r_{A_k+A_k}$.

\section{Proof of Theorem~\ref{thm:half-rep}}
\label{sec:half}

Let $k\ge3$ and $t=|A_k|=(q-1)/k\ge5$.  Recall that, for
$c\in\F_q^\times$, $r_{A_k+A_k}(c)$ counts the ordered pairs
$(x,y)\in A_k^2$ satisfying $x+y=c$.  We first record a consequence of
the minimal-field convention
$f=\ord_t(p)$.

\begin{lemma}
\label{lem:minpower}
Let $\delta\in\{1,2\}$ and let $E\ge1$ be an integer such that
\[
       \delta t+1=p^E .
\]
Then $q=p^f=p^E$ and $k=\delta$.
\end{lemma}

\begin{proof}
Since $t\mid p^E-1$ and $f=\ord_t(p)$, we have $f\mid E$.  If $f<E$, then
$f\le E/2$, and hence
\[
       t\le p^f-1\le p^{E/2}-1.
\]
Therefore
\[
       p^E=\delta t+1\le 2p^{E/2}-1<p^E,
\]
a contradiction.  Thus $f=E$, so $q=p^f=p^E$, and then
\[
       k=\frac{q-1}{t}=\frac{p^E-1}{t}=\delta. \qedhere
\]
\end{proof}

The next lemma converts congruences modulo $t$ into bounds for
$r_{A_k+A_k}(c)$ by forcing every representation $x+y=c$ to give a root of
a fixed low-degree polynomial.

\begin{lemma}
\label{lem:frob}
Let $c\in\F_q^\times$, let $e\geq0$, let $n\geq1$ with $p\nmid n$, and set
$C=c^{p^e}$.
\begin{enumerate}
\item[\textup{(i)}]
If $np^e\equiv1\pmod t$ and $n\geq2$, then
\[
       r_{A_k+A_k}(c)\leq
       \begin{cases}
               n,   & \text{if $n$ is even},\\
               n-1, & \text{if $n$ is odd}.
       \end{cases}
\]

\item[\textup{(ii)}]
If $np^e\equiv-1\pmod t$, then
\[
       r_{A_k+A_k}(c)\leq2n.
\]

\item[\textup{(iii)}]
If $p$ is odd, $np^e=t+2$, and $c\notin A_k$, then
\[
       r_{A_k+A_k}(c)\leq
       \begin{cases}
               2n-2, & \text{if $n$ is even},\\
               2n,   & \text{if $n$ is odd}.
       \end{cases}
\]
\end{enumerate}
\end{lemma}

\begin{proof}
Given a representation $x+y=c$ with $x,y\in A_k$, put
\[
       z=x^{p^e}.
\]
The Frobenius automorphism $u\mapsto u^{p^e}$ is bijective on $\F_q$, so
distinct representations give distinct values of $z$. By its additivity, this gives
\[
       C-z=(c-x)^{p^e}=y^{p^e}.
\]
Since $x,\ y \in A_k$, we have $x^t=y^t=1$, so the congruence class of $np^e$ determines
$z^n=x^{np^e}$ and $(C-z)^n=y^{np^e}$.  It remains, in each case, to
exhibit a nonzero polynomial of the stated degree having every such $z$
as a root.

For \textup{(i)}, we have $z^n=x$ and $(C-z)^n=y$, so $z$ is a root of
\[
       G_1(Z)=Z^n+(C-Z)^n-c.
\]
If $n$ is even, then $p\nmid n$ forces $p$ odd, so the coefficient of $Z^n$ is nonzero and $\deg G_1=n$.  If $n$ is odd, the
leading terms cancel, while the coefficient of $Z^{n-1}$ is $nC\ne0$.
Since $n\ge2$ and $n$ is odd, $n\ge3$, and therefore
$\deg G_1=n-1$.

For \textup{(ii)}, we have $z^n=x^{-1}$ and $(C-z)^n=y^{-1}$, with
$z,C-z\ne0$.  Clearing denominators in
\[
       z^{-n}+(C-z)^{-n}=c
\]
shows that $z$ is a root of
\[
       G_2(Z)=Z^n+(C-Z)^n-cZ^n(C-Z)^n.
\]
The leading coefficient of $G_2$ is $(-1)^{n+1} c$, which is nonzero.
Thus $\deg G_2 =2n$.

For \textup{(iii)}, we have $z^n=x^2$ and
\[
       (C-z)^n=y^2=(c-x)^2.
\]
Hence
\[
       (C-z)^n-z^n-c^2=-2cx.
\]
Squaring and using $x^2=z^n$, every representation gives a root of
\[
       G_3(Z)=\bigl((C-Z)^n-Z^n-c^2\bigr)^2-4c^2Z^n.
\]
Here the exact equality $np^e=t+2$, rather than merely the corresponding
congruence modulo $t$, gives $C^n=c^{t+2}$ and therefore
\[
       G_3(0)=c^4(c^t-1)^2.
\]
This is nonzero because $c\notin A_k$.  Thus $G_3$ is not the zero polynomial.  Its
degree is at most $2n$, and if $n$ is even, the leading terms of
$(C-Z)^n-Z^n$ cancel, giving $\deg G_3 \le 2n-2$.
\end{proof}

\begin{corollary}
\label{cor:tplus1}
If $p\mid t+1$, then
\[
       r_{A_k+A_k}(c)<\frac t2
       \qquad(c\in\F_q^\times).
\]
\end{corollary}

\begin{proof}
Write $t+1=np^e$, with $p\nmid n$, $e\ge1$.  If $n=1$, then Lemma~\ref{lem:minpower} with
$\delta=1$, gives $k=1$, contrary to $k\ge3$.  Thus $n\ge2$. By Lemma~\ref{lem:frob}\textup{(i)},
\[
       r_{A_k+A_k}(c)\le n\le \frac{t+1}{p}.
\]
If $p\ge3$, then
\[
       r_{A_k+A_k}(c)\le\frac{t+1}{3}<\frac t2
\]
because $t>2$.  If $p=2$, then $n$ is odd, so
Lemma~\ref{lem:frob}\textup{(i)} gives
\[
       r_{A_k+A_k}(c)\le n-1\le \frac{t+1}{2}-1<\frac t2. \qedhere
\]
\end{proof}

We next relate $r_{A_k+A_k}(c)$ to point counts on Fermat curves.  For
$c\in\F_q^\times$, let $N(c)$ be the number of projective
$\F_q$-points on
\[
       X^k+Y^k=cZ^k,
\]
and let $N_0(c)$ count the points with $XYZ=0$.  The map
\[
       (X,Y,Z)\mapsto \left( \bigl(\frac{X}{Z}\bigr)^k,\bigl(\frac{Y}{Z}\bigr)^k \right)
\]
sends each projective point with $XYZ\ne0$ to a representation
$x+y=c$ with $x,y\in A_k$, and each such representation has exactly
$k^2$ preimages.  Hence
\begin{equation}
\label{eq:r-fermat}
       r_{A_k+A_k}(c)=\frac{N(c)-N_0(c)}{k^2}.
\end{equation}
For odd $p$, the boundary points are as follows: the line $Z=0$ carries
$k$ points precisely when $-1\in A_k$, equivalently when $t$ is even,
and each of the lines $X=0$ and $Y=0$ carries $k$ points precisely when
$c\in A_k$.

We use the following two estimates of Garcia and Voloch \cite{gv}, obtained
from the St\"ohr--Voloch theory of rational points on curves over finite
fields \cite{stohrvoloch}.

\begin{theorem}[{\cite[Theorems 1 and 2]{gv}}]
\label{gv1}
Suppose $p\ne2$, and let $c\in\F_q^\times$.  If
\[
       k\ne \frac{p^f-1}{p^d-1}
\]
for every proper divisor $d\mid f$, then the number $N(c)$ of
projective points on
\[
       X^k+Y^k=cZ^k
\]
satisfies
\begin{equation}
\label{eq:GV1point}
       N(c)\le \frac{k(q+k-1)-N_0(c)(k-2)}{2}.
\end{equation}
\end{theorem}

\begin{theorem}[{\cite[Theorems 1 and 3]{gv}}]
\label{gvtheorem}
Suppose $p\ge7$, $k\ge3$, and $c\ne0$, and assume that it is not
simultaneously the case that $q=k+1$ and $c=2$.  Then the number $N(c)$
of projective points on
\[
       X^k+Y^k=cZ^k
\]
satisfies
\begin{equation}
\label{eq:GV2point}
       N(c)\le \frac25 k(q+2k-1)-\frac15 N_0(c)(2k-5)
\end{equation}
in each of the following cases:
\begin{enumerate}
\item[\textup{(i)}] $p\nmid (k-1)(k-2)(2k-1)$;

\item[\textup{(ii)}] $p\mid k-2$, unless
\[
       t=\frac{p^d-1}{2}
\]
for some $d\mid f$ with $c\in\F_{p^d}$;

\item[\textup{(iii)}] $p\mid 2k-1$, unless
\[
       t=2(p^d-1)
\]
for some $d\mid f$ with $f/d$ even and $c^2\in\F_{p^d}$.
\end{enumerate}
\end{theorem}

\begin{remark}
\label{rem:gv-scope}
We note that \eqref{eq:GV1point} is originally stated in \cite[Theorem~1]{gv}
for $k\geq 4$, and \eqref{eq:GV2point} stated for
$k\geq 5$.  The cases of $k=3$ and $k=3,4$, respectively,
follow from the same argument applied to the general St\"ohr--Voloch
inequality \cite[p.~346, (*)]{gv}. The corresponding boundary order
sequences are
$(0,1,3),\ (0,1,2,3,4,6),\ (0,1,2,4,5,8)$,
and substitution gives exactly
\eqref{eq:GV1point} and \eqref{eq:GV2point}.
The additional hypothesis excluding $q=k+1$, $c=2$ is due to
Arakelian and Borges \cite[Theorem~2.8(iv) and
Remark~2.9]{arakelianborges}, who identify the further exceptional
case $q=k+1$, $a+b=1$ omitted in \cite[Theorem~3]{gv}. It is vacuous
here, since $q=kt+1$ and $t\ge5$.
\end{remark}

\begin{corollary}
\label{cor:gv}
Let $p$ be odd and let $c\in\F_q^\times$.
\begin{enumerate}
\item[\textup{(a)}] One has
\begin{equation}
\label{eq:GV1}
       r_{A_k+A_k}(c)\le \frac t2+\frac{k-N_0(c)}{2k}.
\end{equation}
In particular, if $c\in A_k$, then $r_{A_k+A_k}(c)<t/2$.

\item[\textup{(b)}] If $p\ge7$ and \eqref{eq:GV2point} holds for $c$,
then $r_{A_k+A_k}(c)<t/2$.
\end{enumerate}
\end{corollary}

\begin{proof}
For \textup{(a)}, the hypothesis of Theorem~\ref{gv1} holds automatically.
Indeed, if
\[
       k=\frac{p^f-1}{p^d-1}
\]
for some proper divisor $d\mid f$, then
\[
       t=\frac{p^f-1}{k}=p^d-1,
\]
so $\ord_t(p)\le d<f$, contrary to $f=\ord_t(p)$.

Combining \eqref{eq:GV1point} with \eqref{eq:r-fermat}, we obtain
\[
       r_{A_k+A_k}(c)
       \le
       \frac{k(q+k-1)-N_0(c)(k-2)-2N_0(c)}{2k^2}
       =
       \frac{k(q+k-1)-kN_0(c)}{2k^2}.
\]
Since $q=kt+1$, this becomes
\[
       r_{A_k+A_k}(c)\le \frac t2+\frac{k-N_0(c)}{2k}.
\]
If $c\in A_k$, then the boundary points with $X=0$ or $Y=0$ contribute $2k$ points to $N_0(c)$, and hence
\[
       r_{A_k+A_k}(c)\le \frac t2-\frac12<\frac t2.
\]

For \textup{(b)}, combining \eqref{eq:GV2point} with
\eqref{eq:r-fermat} gives
\[
\begin{aligned}
       r_{A_k+A_k}(c)
       &\le
       \frac{\frac25 k(q+2k-1)-\frac15 N_0(c)(2k-5)-N_0(c)}{k^2}  \\
       &=\frac25\left(t+2-\frac{N_0(c)}k\right).
\end{aligned}
\]
If $t$ is even, then $-1\in A_k$, so the points at infinity contribute
at least $k$ points to $N_0(c)$.  Therefore
\[
       r_{A_k+A_k}(c)\le \frac25(t+1)<\frac t2,
\]
since $t\ge5$.  If $t$ is odd, then
\[
       r_{A_k+A_k}(c)\le \frac25(t+2)<\frac{t+1}{2}.
\]
As $r_{A_k+A_k}(c)$ is an integer and $t$ is odd, this implies
\[
       r_{A_k+A_k}(c)\le \frac{t-1}{2}<\frac t2. \qedhere
\]
\end{proof}

We treat $p=2$, $p\ge7$, and $p\in\{3,5\}$ separately.  For $p\ge7$ we
apply the Garcia--Voloch estimate and handle its exceptional cases directly.
For $p=3$ and $p=5$, the result follows from elementary congruences and
Lemma~\ref{lem:frob}.

\begin{proof}[Proof of Theorem~\ref{thm:half-rep}]
Fix $c\in\F_q^\times$.

First suppose $p=2$.  Since $t\mid 2^f-1$, the integer $t$ is odd.
Thus $2\mid t+1$, and Corollary~\ref{cor:tplus1} gives
$r_{A_k+A_k}(c)<t/2$.

Next suppose $p\ge7$.  If $p\mid k-1$, then reducing  $kt=q-1$ modulo $p$ gives
\[
       t\equiv -1\pmod p.
\]
Thus $p\mid t+1$, and Corollary~\ref{cor:tplus1} applies.  Now we may assume $p\nmid k-1$, and so at least one of the three cases of Theorem~\ref{gvtheorem} must occur. We first note that unless one  of the exceptional branches \textup{(ii)} or \textup{(iii)} occurs,  Theorem~\ref{gvtheorem} gives \eqref{eq:GV2point}, and therefore
$r_{A_k+A_k}(c)<t/2$ by Corollary~\ref{cor:gv}\textup{(b)}.

Next, suppose the exceptional branch in \textup{(ii)} occurs.  Then
\[
       t=\frac{p^d-1}{2}
\]
for some $d\mid f$.  Thus $2t+1=p^d$. Here,  Lemma~\ref{lem:minpower} with
$\delta=2$ gives $k=2$, contrary to $k\ge3$, and we conclude this branch cannot
occur.

Finally, suppose the exceptional branch in \textup{(iii)} occurs.  Then
$t=2(p^d-1)$ for some $d\mid f$.
Put $K=\F_{p^d}$.  Then $|K^\times|=t/2$.  Since $\F_q^\times$ is cyclic,
$K^\times$ is the unique subgroup of $\F_q^\times$ of order $t/2$, and
therefore $K^\times\le A_k$.  Choosing $\alpha\in A_k\setminus K^\times$, we have
\[
       A_k=K^\times\cup \alpha K^\times,
       \qquad \alpha^2\in K^\times,
       \qquad \alpha\notin K.
\]
Since $p$ is odd, $[K(\alpha):K]=2$.  Moreover, $K(\alpha)$ contains
$A_k$, which generates $\F_q$, and hence
\[
       \F_q=K(\alpha)=K\oplus \alpha K.
\]
It follows that for every $c = x+ y  \in A_k + A_k$,
\[
r_{A_k+A_k}(x+y)=
\begin{cases}
       |K| -2,&x, y \in K^\times \\
       |K| -2,&x, y \in \alpha K^\times  \\
       2,&x \in K^\times , y \in \alpha K^\times.
\end{cases}
\]
Thus,
\[
       r_{A_k+A_k}(c)\le \max(|K|-2,2).
\]
Since $t=2(|K|-1)\ge5$, we have $t\ge6$, and so
\[
       \max(|K|-2,2)=|K|-2=\frac t2-1<\frac t2.
\]
This proves the result for $p\ge7$.

It remains to consider $p\in\{3,5\}$.  Since $t\mid p^f-1$, we have
$p\nmid t$.  If $p\mid t+1$, then Corollary~\ref{cor:tplus1} applies,
so assume $p\nmid t+1$.

Let
\[
       P=
       \begin{cases}
       9,&p=3,\\
       5,&p=5.
       \end{cases}
\]
For $p=3$, the remaining possibilities are
\[
       t\equiv 1,4,7\pmod 9,
\]
and these give respectively
\[
       9\mid t-1,\qquad 9\mid 2t+1,\qquad 9\mid t+2.
\]
For $p=5$, the remaining possibilities are
\[
       t\equiv 1,2,3\pmod 5,
\]
and these give respectively
\[
       5\mid t-1,\qquad 5\mid 2t+1,\qquad 5\mid t+2.
\]
Thus, in all cases, exactly one of
\[
       P\mid t-1,\qquad P\mid 2t+1,\qquad P\mid t+2
\]
holds.  Write the corresponding divisible quantity as $np^e$, with
$p\nmid n$.  Then $p^e\ge P\ge5$.

If
\[
       t-1=np^e,
\]
then $np^e\equiv -1\pmod t$, and Lemma~\ref{lem:frob}\textup{(ii)}
gives
\[
       r_{A_k+A_k}(c)\le 2n\le \frac{2(t-1)}{5}<\frac t2.
\]
If
\[
       2t+1=np^e,
\]
then $np^e\equiv 1\pmod t$.  If $n=1$, then
Lemma~\ref{lem:minpower}, with $\delta=2$, gives $k=2$, contrary to
$k\ge3$.  Hence $n\ge2$, and Lemma~\ref{lem:frob}\textup{(i)} gives
\[
       r_{A_k+A_k}(c)\le n\le \frac{2t+1}{5}<\frac t2.
\]
Finally suppose
\[
       t+2=np^e.
\]
If $c\in A_k$, then Corollary~\ref{cor:gv}\textup{(a)} gives
$r_{A_k+A_k}(c)<t/2$.  If $c\notin A_k$, then
Lemma~\ref{lem:frob}\textup{(iii)} gives
\[
       r_{A_k+A_k}(c)\le 2n\le \frac{2(t+2)}{p^e}.
\]
When $p^e\ge9$, this is already $<t/2$.  When $p^e=5$, it is $<t/2$
for $t>8$.  The only remaining value with $t\ge5$ and $5\mid t+2$ is
$t=8$.  In that case $n=2$ is even, so Lemma~\ref{lem:frob}\textup{(iii)}
gives
\[
       r_{A_k+A_k}(c)\le 2n-2=2<4=\frac t2.
\]
This exhausts all cases and proves the theorem.
\end{proof}

\section{Proof of Corollary~\ref{cor:diff-bound} and Theorem~\ref{thm:no-two-point}}
\label{sec:no-two-point}

For $k \geq 3$, we rule out two-point summands in a hypothetical nontrivial decomposition $A_k = S+T$ using the strict difference bound in Corollary~\ref{cor:diff-bound}, which we prove first. The case $k=2$ requires
a separate argument, as $r_{A_k-A_k}(d)$ can attain the value $|A_k|/2$. Throughout, let $R=A_2\le\F_q^\times$ denote the subgroup of quadratic residues in odd
characteristic, and let $\eta$ be the quadratic character.  The following
identity is well known, but we include a short proof for convenience.

\begin{lemma}
\label{lem:quadratic-difference}
For $d\ne0$,
\[
       r_{R-R}(d)=\frac{q-3-\eta(d)-\eta(-d)}4.
\]
Writing $|R|=(q-1)/2$, one has $r_{R-R}(d)<|R|/2$ except when
$q\equiv1\pmod4$ and $d\notin R$, in which case $r_{R-R}(d)=|R|/2$.
\end{lemma}

\begin{proof}
For $x\ne0$ one has $1_R(x)=(1+\eta(x))/2$.  Expanding
\[
       r_{R-R}(d)=\sum_x1_R(x)1_R(x+d)
\]
and correcting for $x=0,-d$ gives the formula, using
\[
       \sum_x\eta(x)\eta(x+d)=-1
       \qquad(d\ne0). \qedhere
\]
\end{proof}
\begin{proof}[Proof of Corollary~\ref{cor:diff-bound}]
Fix $d\ne0$.  If $-1\in A_k$, then $-A_k=A_k$, so
$r_{A_k-A_k}(d)=r_{A_k+A_k}(d)$, and Theorem~\ref{thm:half-rep} applies.

Suppose $-1\notin A_k$.  Then $q$ and $t$ are odd and $k$ is even; in
particular, $k\ge4$.  Put $B=A_k\sqcup(-A_k)$.  This is the subgroup of
order $2t$ and index $k/2\ge2$, and it still generates $\F_q$.
Partitioning representations according to the two cosets gives
\[
       r_{B+B}(d)=r_{A_k+A_k}(d)+r_{A_k+A_k}(-d)
       +2r_{A_k-A_k}(d).
\]
If $B$ has index at least $3$, then Theorem~\ref{thm:half-rep} gives
$r_{B+B}(d)<|B|/2=t$, and hence $r_{A_k-A_k}(d)<t/2$.  If $B$ has index
$2$, then it is the quadratic-residue subgroup and $B=-B$.  Thus by
Lemma~\ref{lem:quadratic-difference},
\[
       r_{B+B}(d)=r_{B-B}(d)\le \frac{|B|}2=t.
\]
It follows that $2r_{A_k-A_k}(d)\le t$, and because $t$ is odd,
$r_{A_k-A_k}(d)<t/2$.
\end{proof}
\begin{lemma}
\label{lem:nonsquare-progression}
Suppose $q\equiv1\pmod4$ and $q\ge13$.  If $d\notin R$, then $R$ contains a
three-term arithmetic progression with common difference $d$.
\end{lemma}

\begin{proof}
Write the middle term as $y=dz$.  Since $d$ is a nonsquare, the conditions $y-d,\ y,\ y+d\in R$ are equivalent to $z-1,z,z+1$
all being nonzero nonsquares.  For $w\ne0$ the nonsquare indicator is
$(1-\eta(w))/2$, so the number $M$ of admissible $z$ is
\[
      M = \frac18\sum_{z\in E}
       \bigl(1-\eta(z-1)\bigr)\bigl(1-\eta(z)\bigr)\bigl(1-\eta(z+1)\bigr),
\]
where $E=\F_q\setminus\{0,\pm1\}$.  Expanding gives
\[
8M=\sum_{z\in E}\{1-A(z)+B(z)-C(z)\},
\]
where
\[
\begin{aligned}
A(z)&=\eta(z-1)+\eta(z)+\eta(z+1),\\
B(z)&=\eta(z(z-1))+\eta((z-1)(z+1))+\eta(z(z+1)),\\
C(z)&=\eta(z(z-1)(z+1)).
\end{aligned}
\]
Note that 
\[
\sum_{z \in \F_q} A(z) = 0 ,\quad \sum_{z \in F_q} B(z) = 3,
\]
and since $\eta(-1)=1$, we get
\[
\sum_{z\in E} A(z)=-(4+2\eta(2)),
\qquad
\sum_{z\in E} B(z)=-(4+2\eta(2)).
\]
Thus the $A$ and $B$ terms cancel, and since $C(z)=0$ for $z \in \{0, \pm 1\}$, 
\[
8M=q-3-\sum_{z\in\F_q}\eta(z(z-1)(z+1)).
\]
The cubic $z(z-1)(z+1)$ is squarefree, so the Weil bound (see, e.g., \cite{ln}) gives
\[
|\sum_{z\in \F_q} \eta(z(z-1)(z+1))|\le2\sqrt q,
\]
and $M>0$ for $q\ge13$.
\end{proof}

\begin{lemma}
\label{lem:quadratic-no-two-point}
Let $q>9$ be odd.  Then $R=A_2\le\F_q^\times$ has no two-point additive
summand.
\end{lemma}

\begin{proof}
Suppose, after translating the summand, that $R=\{0,d\}+X$, $d \neq 0$. Then
\[
       X\subseteq R\cap(R-d),
       \qquad |R|\le2|X|.
\]
The inclusions above imply $r_{R-R}(d)\ge |R|/2$.  Hence
Lemma~\ref{lem:quadratic-difference} gives
$q\equiv1\pmod4$, $d\notin R$, and $r_{R-R}(d)=|R|/2$.  Consequently,
\[
       X=R\cap(R-d),
       \qquad |X|=\frac{|R|}{2},
\]
and therefore
\[
       R=X\sqcup(X+d).
\]
This partition precludes a three-term arithmetic progression in $R$ with
common difference $d$.  Indeed, if $x,x+d,x+2d\in R$, then $x,x+d\in X$,
while $x+d\in X+d$, contradicting the disjointness.
Lemma~\ref{lem:nonsquare-progression} now gives the required contradiction.
\end{proof}
\begin{proof}[Proof of Theorem~\ref{thm:no-two-point}]
Let $A_k=S+T$ be a nontrivial decomposition.

Suppose first that $k=2$, i.e., $A_k=R$.  Since $t\ge5$, one has $q>9$,
and Lemma~\ref{lem:quadratic-no-two-point} shows that neither summand can
have cardinality two.  Moreover, for distinct $a,b\in S$,
\[
       T\subseteq(R-a)\cap(R-b),
\]
whose size is $r_{R-R}(a-b)\le t/2$ by
Lemma~\ref{lem:quadratic-difference}.  Interchanging $S$ and $T$ gives
the same bound for $S$.  Hence $3\le|S|,|T|\le t/2$.

Assume now that $k\ge3$.  For distinct $a,b\in S$,
\[
       T\subseteq(A_k-a)\cap(A_k-b),
\]
whose size is $r_{A_k-A_k}(a-b)<t/2$ by
Corollary~\ref{cor:diff-bound}; by symmetry, $|S|<t/2$ as well.
Finally, Corollary~\ref{cor:diff-bound} verifies the hypothesis of
Proposition~\ref{prop:edge-obstruction}, which rules out a two-point
summand, so $|S|,|T|\ge3$.
\end{proof}

\section{Proof of Theorem~\ref{thm:proper-classification}}
\label{sec:proper}

The cases of orders $1$, $2$, and $4$ have explicit GAP representations.
The subgroup of order $1$ is a rank-zero GAP, and a subgroup of order $2$ is $\{\pm1\}=-1+\{0,2\}$. If $A_k$ has order $4$, then for a generator $i$ one has $i^2=-1$, and
\[
       A_k=\{\pm1,\pm i\}
       =-1+\{0,1-i\}+\{0,1+i\}.
\]

It remains only to exclude order $3$ and orders at least $5$.

\begin{lemma}
\label{lem:three-point-gap}
Every three-element GAP is a three-term arithmetic progression.
\end{lemma}

\begin{proof}
By Lemma~\ref{lem:gap-two-point}, write the GAP, after translation, as
\[
       P=T+\{0,d\},
       \qquad d\ne0.
\]
Since $T\subseteq P$, one has $2\le|T|\le3$.

If $|T|=2$, then $|T\cup(T+d)|=3$, so $T$ and $T+d$ share exactly one point. Since $d \neq 0$, this forces $T=\{x,x+d\}$ for some $x$, and then $P = \{x,\ x+d,\ x+2d\}$, a three-term progression.

If $|T|=3$, then $T=P$ and $P+d=P$.  Translation by $d$ cyclically permutes
the three points, so again $P=\{x,x+d,x+2d\}$.
\end{proof}

\begin{lemma}
\label{lem:order-three}
A proper multiplicative subgroup of order $3$ is not a GAP.
\end{lemma}

\begin{proof}
In characteristic $2$, the minimal field containing a subgroup of order $3$ is
$\F_4$, where it is the full multiplicative group, so no proper subgroup of order $3$ occurs. In characteristic $3$ no multiplicative subgroup has order $3$. Thus we may assume
$p \ge 5$.

Let
\[
       A_k=\{1,\omega,\omega^2\},
       \qquad \omega^3=1,
       \quad \omega\ne1.
\]
If $A_k$ were a GAP, Lemma~\ref{lem:three-point-gap} would give distinct
$a,b,c\in A_k$ satisfying $a+c=2b$.  Since
\[
       a+b+c=1+\omega+\omega^2=0,
\]
it follows that $3b=0$, which is impossible because $b\ne0$ and $p\ne3$.
\end{proof}

\begin{proof}[Proof of Theorem~\ref{thm:proper-classification}]
The orders $1$, $2$, and $4$ give the displayed examples.  The case
$|A_k|=3$ is excluded by Lemma~\ref{lem:order-three}.  If $|A_k|\ge5$ and
$A_k$ were a GAP, Lemma~\ref{lem:gap-two-point} would give a two-point
summand, contrary to Theorem~\ref{thm:no-two-point}.
\end{proof}
\section{Proof of Theorem~\ref{thm:full-groups}}
\label{sec:full}

We first isolate the elementary obstruction to proper representations.

\begin{proposition}
\label{prop:punctured-space}
Let $V$ be an $\F_p$-vector space of dimension greater than one, and let
$b\in V$.  Then $V\setminus\{b\}$ is not a proper GAP.
\end{proposition}

\begin{proof}
Suppose
\[
       V\setminus\{b\}=a+P_1+\cdots+P_r,
       \qquad
       P_i=[0,L_i-1]d_i,
\]
with every element represented uniquely; here $r\ge1$, since
$|V\setminus\{b\}|=|V|-1>1$.

Fix an index $i$ and let $U=\F_pd_i$ be the line through $d_i$.  Unique
representation, applied already within the single progression $P_i$,
shows that $P_i$ has exactly $L_i$ distinct elements, all in $U$.
Grouping the remaining summands, write
\[
       V\setminus\{b\}
       =\bigcup_{c\,\in\,a+\sum_{j\ne i}P_j}(c+P_i),
\]
where each translate $c+P_i$ consists of exactly $L_i$ points of a single
coset of $U$.  Properness says precisely that this union is disjoint: the
translates of $P_i$ partition $V\setminus\{b\}$. Since $P_i\subseteq U$, each translate lies in a single coset:

\[
       c+P_i\subseteq c+U
       \qquad\Bigl(c\in a+\sum_{j\ne i}P_j\Bigr).
\]
Hence the number of points of $V\setminus\{b\}$ in any coset of $U$ is a
multiple of $L_i$.  Now the coset $b+U$ meets $V\setminus\{b\}$ in $p-1$
points, while every other coset meets it in $p$ points, and both kinds of
coset occur because $\dim V>1$.  Thus $L_i$ divides both $p$ and $p-1$,
impossible because $L_i\ge2$.
\end{proof}

\begin{proof}[Proof of Theorem~\ref{thm:full-groups}]
Suppose first that $p$ is odd.  If $f=1$, then
$\F_p^\times=1+[0,p-2]$ is an arithmetic progression.  Assume $f>1$ and
identify the additive group of $\F_{p^f}$ with $V=\F_p^f$, so that
$\F_q^\times$ becomes the punctured space $V\setminus\{0\}$.  By
Proposition~\ref{prop:punctured-space}, $V\setminus\{0\}$ is not a proper
GAP.

It remains to construct an improper representation. The idea
is to begin with a box in which each coordinate omits only the value
$-1$, and then use two-point summands to fill every remaining hole
except one. Let
$I=\F_p\setminus\{-1\} = [0, p-2]$, let $e_1,\ldots,e_f$ be the standard basis, and
put
\[
       u=e_1+\cdots+e_{f-1}+(f-1)e_f,
       \qquad w_i=e_i-e_f\quad(1\le i<f).
\]
Note $u, w_i \neq 0$. We claim that
\begin{equation}
\label{eq:full-group-improper}
       Ie_1+\cdots+Ie_f+\{0,u\}+\sum_{i<f}\{0,w_i\}
       =V\setminus\{-e_f\}.
\end{equation}

Write $I^f=Ie_1+\cdots+Ie_f$, and let $D$ denote the sum of the
two-point summands, so that every $\delta\in D$ has the form
\[
       \delta
       =\varepsilon u+\eta_1w_1+\cdots+\eta_{f-1}w_{f-1},
       \qquad
       \varepsilon,\eta_i\in\{0,1\},
\]
with coordinates
\[
       \delta_i=\varepsilon+\eta_i\quad(i<f),
       \qquad
       \delta_f=(f-1)\varepsilon-\sum_{i<f}\eta_i.
\]
For $x=(x_1,\ldots,x_f)$, put $y_i=x_i+1$.  Since each coordinate of
$I^f$ takes every value except $-1$, we have $x\in I^f+D$ precisely when
some $\varepsilon,\eta_i\in\{0,1\}$ satisfy $\delta_i\ne y_i$ for all
$i$, that is,
\begin{equation}
\label{eq:avoid}
       \varepsilon+\eta_i\ne y_i\quad(i<f),
       \qquad
       (f-1)\varepsilon-\sum_{i<f}\eta_i\ne y_f.
\end{equation}

Fix $\varepsilon$. For each $i<f$, the condition $\varepsilon + \eta_i \neq y_i$ excludes at most one value of $\eta_i$.  If both choices of some $\eta_i$ remain admissible, we choose it last. Then its two values
give distinct last coordinates, so one satisfies the condition on $y_f$.
Failure can therefore occur only when every $\eta_i$ is forced, equivalently
when $y_i\in\{\varepsilon,\varepsilon+1\}$ for all $i<f$.  In that case
$\eta_i=1+\varepsilon-y_i$, and the forced last coordinate is
\[
       (f-1)\varepsilon-\sum_{i<f}\eta_i
       =\sum_{i<f}y_i-(f-1),
\]
which is independent of $\varepsilon$.

If \eqref{eq:avoid} fails for both $\varepsilon=0$ and $\varepsilon=1$,
then $y_i\in\{0,1\}\cap\{1,2\}=\{1\}$ for every $i<f$ (here $p$ is odd), and the forced
last coordinate is $0$.  Hence $y_f=0$ and $x=-e_f$.  Conversely, if $x = -e_f$, then $y_i = 1$ for all $i<f$ and $y_f = 0$, so for each $\varepsilon$ every $\eta_i$ is forced and the forced last coordinate is $0 =  y_f$. Thus \eqref{eq:avoid} fails and $-e_f$ is missed. This proves
\eqref{eq:full-group-improper}.  Translating by $e_f$ gives a GAP
representation of $V\setminus\{0\}$ with $2f$ component progressions.

Finally, let $p=2$.  Every nontrivial component progression is a two-point
set $\{0,d_i\}$, so a positive-rank GAP is a coset of an $\F_2$-subspace
and has power-of-two cardinality.  Since
$|\F_{2^f}^\times|=2^f-1$, the full multiplicative group is a GAP only for
$f=1$, when $\F_2^\times=\{1\}$ is a rank-zero GAP.
\end{proof}

\subsection*{Acknowledgments}
We thank Craig Spencer for many helpful discussions.

\end{document}